\documentclass[10pt]{article}

\usepackage{times}
\usepackage{array, amssymb, amsmath, graphics}
\newtheorem{remark}{Remark}
\newtheorem{theorem}{Theorem}
\newtheorem{example}{Example}

\newtheorem{lemma}{Lemma}

\newenvironment{proof}{{\bf Proof.}}{\hspace*{1mm}\hfill\rule{2mm}{2mm}}

\newtheorem{pretheorema}{{\bf Theorem}}

\newtheorem{prelem}{{\bf Theorem}}

\newtheorem{prelemlem}{{\bf Lemma}}

\def\m#1#2{\raise 0.2ex\hbox{
    ${#1_{\displaystyle #2}}$}}
\def\x#1{\raise 0.5ex\hbox{
    ${#1}$}}
\def\n#1{\vbox to 3mm{\vspace{0mm}\vfill \hbox to 4.5mm{\hfill
             $#1$\hfill} \vfill }}
\def\arraystretch{1.0}                

\def\x{{\bf x}}

\begin{document}
\title{On uniformly generating  {L}atin squares}
\author{M. Aryapoor\footnote{Institute for Studies in Theoretical Physics and Mathematics
(IPM), Niavaran Square, Tehran, Iran ({\tt
masood.aryapoor@ipm.ir}).} \ and E. S.
Mahmoodian\footnote{Department of Mathematical Sciences, Sharif
University of Technology, P. O. Box 11155-9415, Tehran, Iran
({\tt emahmood@sharif.edu}).} }
\date{}
\maketitle
\begin{abstract}
By simulating an ergodic Markov chain whose stationary
distribution is uniform over the space of $n \times n$ Latin
squares, Mark T. Jacobson and Peter Matthews~\cite{MR1410617}, \
have discussed  elegant methods by which they generate {L}atin
squares with a uniform distribution  (approximately). The central
issue is the construction of ``moves'' that connect the squares.
Most of their lengthy paper is to prove that the associated graph
is indeed connected. We give a short proof of this fact by using
 the  concepts of Latin bitrades.

\end{abstract}

\section{Introduction and preliminaries}

A {\sf Latin square} $L$  of order $n$ is an $n\times n$ array
with entries chosen from an $n$-set  $N$, e.g. $\{1,\dots, n\}$,
in such a way that each element of $N$ occurs precisely once in
each row and column of the array.
A {\sf partial Latin square}
 $P$ of order $n$
 is an $n\times n$ array with entries chosen from an $n$-set $N$,
in such a way that each element  of $N$ occurs at most once in
each row and at most once in each column of the array. Hence
there are cells in the array that may be empty, but the positions
that are filled have been so as to conform with the Latin
property of array. For ease of exposition,
 a partial Latin square $T$ may be
represented as a set of ordered triples:  $\{(i,j;T_{ij})\mid
{\rm  where\  element\ } T_{ij} {\rm \ occurs \  in\ } {\rm \
(nonempty)} \ {\rm cell\ } (i,j) {\rm \ of\ the\  array}\}$.

Let $T$ be a partial Latin square and $L$ a Latin square such
that $T\subseteq L$. Then $T$ is called a {\sf Latin trade}, if
there exists a partial Latin square  $T^{*}$ such that $T^{*}\cap
T=\emptyset$ and $(L\backslash T)\cup T^{*}$ is a Latin square.
We call $T^{*}$ a {\sf disjoint mate} of $T$ and the pair
$\mathcal{T}=(T,T^{*})$ is called a {\sf Latin bitrade}. The {\sf
volume} of a Latin bitrade is the number of its nonempty cells. A
Latin bitrade of volume 4 which is unique (up to isomorphism), is
said to be an {\sf intercalate}. A bitrade
$\mathcal{T}=(T,T^{*})$ may be viewed as a set of positive triples
$T$ and   negative triples $T^{*}$.
\begin{example}
\label{intercalate} The  bitrade $\mathcal{I}=(I,I^*)$, where
$$I=\{(i,j;a),(i
,j';b),(i',j;b),(i',j';a)\}, $$
$$I^*=\{(i,j;b),(i
,j';a),(i',j;a),(i',j';b)\},$$

 is an intercalate. Note that we must have $i\neq i'$, $j\neq j'$
 and $a\neq b$.
 Usually such an intercalate is shown as

\begin{center}
$\mathcal{I}$ \qquad \begin{tabular}
{@{\hspace{1pt}}c@{\hspace{1pt}}|@{\hspace{1pt}}c@{\hspace{1pt}}
 @{\hspace{1pt}}c@{\hspace{1pt}}
 @{\hspace{1pt}}c@{\hspace{1pt}}@{\hspace{1pt}}c@{\hspace{1pt}}
 @{\hspace{1pt}}c@{\hspace{1pt}}@{\hspace{1pt}}c@{\hspace{1pt}}
 @{\hspace{1pt}}c@{\hspace{1pt}}@{\hspace{1pt}}c@{\hspace{1pt}}
}
\n{}     &\n{}         &  \n{j}     & \n{}        & \n{j'}  & \n{}  & \n{} \\
\hline
\n{}       &\n{.}   & \n{.}         &  \n{\dots}        &\n{.}         &  \n{.}      \\
\n{i}      &\n{.}   & \m{a}{b}      &  \n{\dots}        & \m{b}{a}     &  \n{.}      \\
\n{ }      &\n{.}   & \n{.}         &  \n{\dots}        &\n{.}         &  \n{.}     \\
\n{i'}     &\n{.}   & \m{b}{a}      &  \n{\dots}        & \m{a}{b}     &  \n{.}         \\
\n{}       &\n{.}   & \n{.}         &  \n{\dots}        & \n{.}        &  \n{.}        \\

\end{tabular}
\end{center}

where the elements of $I^*$ are written as subscripts in the
 same array as $I$.

\end{example}
For a recent survey on Latin bitrades see~\cite{MR2453264} and
also~\cite{MR2048415}.

 In~\cite{MR1410617} the  approach for generating Latin squares is based on the fact that an $n \times n$ Latin square
 is equivalent to an $n \times n \times n$
contingency (proper) table in which each line sum equals 1. They
relax the nonnegativity condition on the table's cells, allowing
``improper'' tables that have a single $-1$-cell. A simple set of
moves connects this expanded space of tables [the diameter of the
associated graph is bounded by $2(n - 1)^3$] and suggests a
Markov chain whose subchain of proper tables has the desired
uniform stationary distribution. By grouping these moves
appropriately, they derive a class of moves that stay within the
space of proper Latin squares.

An {\sf improper Latin square} is an $n \times n$ array such that  each cell has a single
 symbol, except for {\em one} improper cell (in the improper
row and column) which has three (the improper symbol appears
there with a -1 coefficient). Each symbol appears exactly once in
each row and in each column, except in the improper row (and also
in the improper column) where one of the symbols appears twice as
``positive'' and once as ``negative''. An improper Latin square
may be viewed as a set of $n^2+1$ positive triples and one
negative triple.

\begin{example}
\label{improperls4} The following array is an improper Latin
square of order $4$.

\begin{center}
$L \quad $ $\begin{array}
{@{\hspace{3.5pt}}c@{\hspace{3.5pt}}@{\hspace{3.5pt}}c@{\hspace{3.5pt}}
 @{\hspace{3.5pt}}c@{\hspace{3.5pt}}@{\hspace{3.5pt}}c@{\hspace{3.5pt}}
 }
  c    &  b     &   d    &   a    \\
  b    &  d     &   a    &   c    \\
  d    &  a+c-b &   b    &   b    \\
  a    &  b     &   c    &   d    \\
\end{array}$
\end{center}

Using the notation of Latin bitrades, we may show this improper
Latin square  by

\begin{center}
$L \quad $ $\begin{array}
{@{\hspace{3.5pt}}c@{\hspace{3.5pt}}@{\hspace{3.5pt}}c@{\hspace{3.5pt}}
 @{\hspace{3.5pt}}c@{\hspace{3.5pt}}@{\hspace{3.5pt}}c@{\hspace{3.5pt}}
 }
  c    &  b          &   d    &   a    \\
  b    &  d          &   a    &   c    \\
  d    &  a+\m{c}{b} &   b    &   b    \\
  a    &  b          &   c    &   d    \\
\end{array}$
\end{center}

\end{example}

The notion of  {\sf $\pm 1$-move} is introduced   in
~\cite{MR1410617}. Using the notation of Latin bitrades,  a {\sf
$\pm 1$-move}  means adding some appropriate intercalate  to a
given proper or improper Latin square such that the result is a proper or improper Latin square.
If the added intercalate
is the intercalate ${\cal I} $ in Example~\ref{intercalate}, the
corresponding  $\pm 1$-move is called   a {\sf
$((i,j;a),(i',j';b))$-move}.
\begin{example}
\label{ls4} By applying the $((1,2;a),(3,4;b))$-move   to the
improper Latin square $L$ in Example~\ref{improperls4}, we obtain
the following Lain square.
\end{example}

\begin{center}
$L' \quad $ $\begin{array}
{@{\hspace{3.5pt}}c@{\hspace{3.5pt}}@{\hspace{3.5pt}}c@{\hspace{3.5pt}}
 @{\hspace{3.5pt}}c@{\hspace{3.5pt}}@{\hspace{3.5pt}}c@{\hspace{3.5pt}}
 }
  c    &  a     &   d    &   b    \\
  b    &  d     &   a    &   c    \\
  d    &  c     &   b    &   a    \\
  a    &  b     &   c    &   d    \\
\end{array}$
\end{center}

Let $G=(V,E)$ be a graph whose vertices are associated to $S$,
the set of all proper and improper Latin squares of order $n$, and
two vertices $L$ and $L'$ are adjacent if there is a $\pm 1$-move
transferring $L$ to $L'$. In the next section we state the results
which prove that $G$ is connected. This approach is developed
from a linear algebraic approach to the concept of Latin bitrades,
which is detailed in the
references~\cite{KhanbanMahdianMahmoodian},
~\cite{MahmoodinNajafian} and~\cite{MR1880972}.

%
\section{Connectivity of graph $G$}
In this section we prove that the graph $G$ (defined in the last
section) is connected. First we need a few lemmas. The first
lemma states that an improper Latin square can be transferred into
a proper Latin square using $\pm 1$-moves with changes only in
two rows.
\begin{lemma}\label{improper-proper}
Suppose that we have the following improper Latin square
\begin{center}
\def\arraystretch{0.8}
$A$ \qquad
$\begin{array}
{@{\hspace{7pt}}c@{\hspace{7pt}}|@{\hspace{7pt}}c@{\hspace{7pt}}
 @{\hspace{7pt}}c@{\hspace{7pt}}
 @{\hspace{7pt}}c@{\hspace{7pt}}@{\hspace{7pt}}c@{\hspace{7pt}}
}
\n{}     &\n{}     & \n{j}         &\n{} \\ \hline
\n{}     &\n{.}    & \n{.}         &  \n{.}    \\
\n{i_1}  &\n.      & \n{a}+\m{b}{s}&  \n.      \\
\n{}     &\n{.}    & \n{.}         &  \n{.}    \\
\n{i_2}  &\n{.}    & \n{s}         &  \n{.}    \\
\n{}     &\n{.}    & \n{.}         &  \n{.}    \\
\end{array}$
\end{center}
Then there is a sequence of (at most $\frac{n-1}{2}$) $\pm
1$-moves involving only rows $i_1$ and $i_2$ which transfers $A$
to a proper Latin square.
\end{lemma}

\begin{proof}
It is easy to see that we can find the following cyclic pattern
lying in rows $i_1$ and $i_2$ of $A$ (possibly after permuting
some columns of $A$)
\begin{center}
$A$ \quad
\def\arraystretch{0.8}
\begin{tabular}
{@{\hspace{7pt}}c@{\hspace{7pt}}|@{\hspace{7pt}}c@{\hspace{7pt}}
 @{\hspace{7pt}}c@{\hspace{7pt}}
 @{\hspace{7pt}}c@{\hspace{7pt}}
  @{\hspace{7pt}}c@{\hspace{7pt}}
 @{\hspace{7pt}}c@{\hspace{7pt}}
 @{\hspace{7pt}}c@{\hspace{7pt}}@{\hspace{7pt}}c@{\hspace{7pt}}
 @{\hspace{7pt}}c@{\hspace{7pt}}
 @{\hspace{7pt}}c@{\hspace{7pt}}
} \n{}     &\n{}     & \n{j}   &  \n{j_1}     & \n{j_{2}} &
\n{j_{3}} & \n{\dots} & \n{j_{r-1}}   & \n{j_r} &\n{}  \\ \hline
\n{}     &\n{.}    & \n{.}     &  \n{.}    & \n{.}  &\n{.}        & \n{\dots}  &  \n{.}     &\n{.}&\n{.} \\
\n{i_1}  &\n{.}     & \n{a}+\m{b}{s}        &  \n{t}     &\n{u} &\n{v}  & \n{\dots}   &  \n{z}   &\n{s} &\n{.}  \\
\n{}     &\n{.}    & \n{.}     &  \n{.}    & \n{.}        &\n{.}  & \n{\dots}  &  \n{.}     &\n{.}&\n{.} \\
\n{i_2}  &\n{.}    & \n{s}     &  \n{b}    & \n{t}      &\n{u} & \n{\dots}  &  \n{y}    &\n{z} &\n{.} \\
\n{}     &\n{.}    & \n{.}     &  \n{.}    & \n{.}       &\n{.}    & \n{\dots}  &  \n{.}     &\n{.}&\n{.} \\
\end{tabular}
\end{center}
where $t,u,...,z\notin \{s,a,b\}$ or $r=1$ (i.e. $t=s$). Note that there is a
similar pattern corresponding to $a$ which has no intersection
with the above pattern except in the $j$th column. Therefore one
of these patterns is at most of length $\frac{n-1}{2}$, and we
may assume that  $r \le \frac{n-1}{2}$. We proceed by induction
on $r$. If $r=1$, then the
$((i_1,j;s),(i_2,j_1;b)$)-move produces a proper Latin square.
Let $r>1$. Then the $((i_1,j;s),(i_2,j_r;b)$)-move decreases $r$.
\end{proof}

In the next lemma we show that one can swap a cycle lying in two
rows using $\pm 1$-moves. In~\cite{MR2108400} it is called a
cycle switch.

\begin{lemma}\label{cycleswap}

Suppose we have the following cyclic pattern in a (proper) Latin
square
\begin{center}
$A$ \qquad
\def\arraystretch{0.8}
\begin{tabular}
{@{\hspace{7pt}}c@{\hspace{7pt}}|@{\hspace{7pt}}c@{\hspace{7pt}}
 @{\hspace{7pt}}c@{\hspace{7pt}}
 @{\hspace{7pt}}c@{\hspace{7pt}}
 @{\hspace{7pt}}c@{\hspace{7pt}}
 @{\hspace{7pt}}c@{\hspace{7pt}}@{\hspace{7pt}}c@{\hspace{7pt}}
 @{\hspace{7pt}}c@{\hspace{7pt}}
 @{\hspace{7pt}}c@{\hspace{7pt}}
}
\n{}      & \n{}   &  \n{j_1} & \n{j_{2}} & \n{j_{3}}& \n{\dots} & \n{j_{r-1}}& \n{j_r} &\n{} \\ \hline
\n{}      & \n{.}  &  \n{.}   & \n{.}     & \n{.}    & \n{\dots} & \n{.}      &\n{.}    &\n{.}\\
\n{i_1}   & \n{.}  &\n{s}     & \n{t}     & \n{u}    & \n{\dots} & \n{y}      &\n{z}    &\n{.}\\
\n{}      & \n{.}  &  \n{.}   & \n{.}     & \n{.}    & \n{\dots} & \n{.}      &\n{.}    &\n{.}\\
\n{i_2}   & \n{.}  &  \n{t}   & \n{u}     & \n{v}    & \n{\dots} & \n{z}      &\n{s}    &\n{.}\\
\n{}      & \n{.}  &  \n{.}   & \n{.}     & \n{.}    & \n{\dots} & \n{.}      &\n{.}    &\n{.}\\
\end{tabular}
\end{center}
Then there is a sequence of (exactly $r-1$) $\pm 1$-moves acting
only on the  entries  shown above which transfers $A$ to
\begin{center}
\def\arraystretch{0.8}
\begin{tabular}
{@{\hspace{7pt}}c@{\hspace{7pt}}|@{\hspace{7pt}}c@{\hspace{7pt}}
 @{\hspace{7pt}}c@{\hspace{7pt}}
 @{\hspace{7pt}}c@{\hspace{7pt}}
 @{\hspace{7pt}}c@{\hspace{7pt}}
 @{\hspace{7pt}}c@{\hspace{7pt}}
@{\hspace{7pt}}c@{\hspace{7pt}}
 @{\hspace{7pt}}c@{\hspace{7pt}}
 @{\hspace{7pt}}c@{\hspace{7pt}}
}
\n{}   & \n{}   &\n{j_1}   & \n{j_{2}}& \n{j_{3}} & \n{\dots}  & \n{j_{r-1}}& \n{j_r} &\n{} \\ \hline
\n{}   & \n{.}  &\n{.}     & \n{.}    & \n{.}     & \n{\dots}  &  \n{.}     &\n{.}    &\n{.}\\
\n{i_1}& \n{.}  &\n{t}     & \n{u}    & \n{v}     & \n{\dots}  &  \n{z}     &\n{s}    &\n{.}\\
\n{}   & \n{.}  & \n{.}    & \n{.}    & \n{.}     & \n{\dots}  &  \n{.}     &\n{.}    &\n{.}\\
\n{i_2}& \n{.}  & \n{s}    & \n{t}    & \n{u}     & \n{\dots}  &  \n{y}     &\n{z}    &\n{.}\\
\n{}   & \n{.}  & \n{.}    & \n{.}    & \n{.}     & \n{\dots}  &  \n{.}     &\n{.}    &\n{.}\\
\end{tabular}
\end{center}

\end{lemma}

\begin{proof}
If $r=2$ (i.e. $u=s$), then the $((i_1,j_1;t),(i_2,j_2;s)$)-move does the
job. So let $r\ge 3$. Then the $((i_1,j_1;t),(i_2,j_2;s)$)-move
transfers $A$ to
\begin{center}
\def\arraystretch{0.8}
\begin{tabular}
{@{\hspace{7pt}}c@{\hspace{7pt}}|@{\hspace{7pt}}c@{\hspace{7pt}}
 @{\hspace{7pt}}c@{\hspace{7pt}}
 @{\hspace{7pt}}c@{\hspace{7pt}}
  @{\hspace{7pt}}c@{\hspace{7pt}}
 @{\hspace{7pt}}c@{\hspace{7pt}}
 @{\hspace{7pt}}c@{\hspace{7pt}}@{\hspace{7pt}}c@{\hspace{7pt}}
 @{\hspace{7pt}}c@{\hspace{7pt}}
 @{\hspace{7pt}}c@{\hspace{7pt}}
} \n{}           & \n{}   &  \n{j_1}     & \n{j_{2}}  & \n{j_{3}}&
\n{\dots} & \n{j_{r-1}}   & \n{j_r} &\n{}  \\ \hline
\n{}          & \n{.}     &  \n{.}    & \n{.}    &  \n{.}          & \n{\dots}  &  \n{.}     &\n{.}&\n{.}  \\
\n{i_1}        & \n{.}  &  \n{t}     & \n{s} &  \n{u} & \n{\dots}   &  \n{y}   &\n{z}         &\n{.}   \\
\n{}         & \n{.}     &  \n{.}    & \n{.}   &  \n{.}           & \n{\dots}  &  \n{.}     &\n{.}&\n{.}  \\
\n{i_2}       & \n{.}     &  \n{s}    & \n{u}+\m{t}{s} &  \n{v}  & \n{\dots}  &  \n{t}    &\n{s} &\n{.}  \\
\n{}         & \n{.}     &  \n{.}    & \n{.}     &  \n{.}         & \n{\dots}  &  \n{.}     &\n{.}&\n{.}  \\
\end{tabular}
\end{center}
Now by applying the method in the proof of  Lemma~\ref{improper-proper}, this improper Latin square can be
transferred to the desired Latin square.
\end{proof}
%

The last lemma is a crucial lemma. It tells us that we can switch two
entries in a row of an improper Latin square using a sequence of controlled $\pm 1$-moves.

\begin{lemma}\label{ablemma}
Suppose for given $s$ and $t$ $ \in \{1,2, \ldots, n\}$ we have the following improper Latin square:
\begin{center}
$A$ \qquad
\def\arraystretch{0.8}
$\begin{array}
{@{\hspace{7pt}}c@{\hspace{7pt}}|@{\hspace{7pt}}c@{\hspace{7pt}}
 @{\hspace{7pt}}c@{\hspace{7pt}}
 @{\hspace{7pt}}c@{\hspace{7pt}}
 @{\hspace{7pt}}c@{\hspace{7pt}}@{\hspace{7pt}}c@{\hspace{7pt}}
}
\n{}     &\n{}     & \n{j_1}        &  \n{}     & \n{j_2} & \n{} \\ \hline
\n{}     &\n{.}    & \n{.}          &  \n{.}    & \n{.}   & \n{.}\\
\n{i_1}  &\n.      & \n{s}          &  \n.      & \n{t}   & \n.  \\
\n{}     &\n{.}    & \n{.}          &  \n{.}    & \n{.}   & \n{.}\\
\n{i_2}  &\n{.}     & \n{a}+\m{b}{s}&  \n{.}    & \n{.}   & \n{.}\\
\n{}     &\n{.}    & \n{.}          &  \n{.}    & \n{.}   & \n{.}\\
\n{i_3}  &\n{.}    & \n{.}          &  \n{.}    & \n{s}   & \n{.}\\
\end{array}$
\end{center}
where $i_2$ may be equal to $i_3$. Then there is a sequence of
(at most $2(n-1)$) $\pm 1$-moves transferring this square to an
improper (or proper) Latin square $A'$ of the following form:
\begin{center}
$A'$ \qquad
\def\arraystretch{0.8}
\begin{tabular}
{@{\hspace{7pt}}c@{\hspace{7pt}}|@{\hspace{7pt}}c@{\hspace{7pt}}
 @{\hspace{7pt}}c@{\hspace{7pt}}
 @{\hspace{7pt}}c@{\hspace{7pt}}
 @{\hspace{7pt}}c@{\hspace{7pt}}@{\hspace{7pt}}c@{\hspace{7pt}}
}
\n{}   &\n{}     & \n{j_1}       &  \n{}     & \n{j_2} & \n{}  \\ \hline
\n{}   &\n{.}    & \n{.}         &  \n{.}    & \n{.}   & \n{.} \\
\n{i_1}&\n.      & \n{t}         &  \n.      & \n{s}   & \n.   \\
\n{}   &\n{.}    & \n{.}         &  \n{.}    & \n{.}   & \n{.} \\
\n{i'} &\n{.}    & \n{e}+\m{f}{t}&  \n{.}    & \n{.}   & \n{.} \\
\n{}   &\n{.}    & \n{.}         &  \n{.}    & \n{.}   & \n{.} \\
\end{tabular}
\end{center}
where $i'=i_2$ or $i_3$, and  the only possibly different entries
of $A$ and $A'$ are entries in: $(i_1,j_1)$, $(i_1,j_2)$ and
those in rows $i_2$ and $i_3$.
\end{lemma}
\begin{proof}
We distinguish two cases.

{\bf Case 1}:   $i_2=i_3$, i.e. in column $j_2$ the symbol $s$
appears in the improper row. So $A$ has the following form:
\begin{center}
$A$ \qquad
\def\arraystretch{0.8}
\begin{tabular}
{@{\hspace{7pt}}c@{\hspace{7pt}}|@{\hspace{7pt}}c@{\hspace{7pt}}
 @{\hspace{7pt}}c@{\hspace{7pt}}
 @{\hspace{7pt}}c@{\hspace{7pt}}
  @{\hspace{7pt}}c@{\hspace{7pt}}
 @{\hspace{7pt}}c@{\hspace{7pt}}
 @{\hspace{7pt}}c@{\hspace{7pt}}@{\hspace{7pt}}c@{\hspace{7pt}}
} \n{}     &\n{}     & \n{j_1}   &  \n{}     & \n{j_2} & \n{}  \\ \hline
\n{}     &\n{.}    & \n{.}     &  \n{.}    & \n{.}              & \n{.}   \\
\n{i_1}  &\n.      & \n{s}     &  \n.      & \n{t}           & \n.   \\
\n{}     &\n{.}    & \n{.}     &  \n{.}    & \n{.}              & \n{.}  \\
\n{i_2}  &\n{.}     & \n{a}+\m{b}{s}  &  \n{.}     & \n{s}  & \n{.}  \\
\n{}     &\n{.}    & \n{.}     &  \n{.}    & \n{.}              & \n{.}  \\
\end{tabular}
\end{center}
Then the  $((i_1,j_1;t),(i_2,j_2;s))$-move  transfers $A$ to:
\begin{center}
$A'$ \qquad
\def\arraystretch{0.8}
\begin{tabular}
{@{\hspace{7pt}}c@{\hspace{7pt}}|@{\hspace{7pt}}c@{\hspace{7pt}}
 @{\hspace{7pt}}c@{\hspace{7pt}}
 @{\hspace{7pt}}c@{\hspace{7pt}}
  @{\hspace{7pt}}c@{\hspace{7pt}}
 @{\hspace{7pt}}c@{\hspace{7pt}}
 @{\hspace{7pt}}c@{\hspace{7pt}}@{\hspace{7pt}}c@{\hspace{7pt}}
} \n{}     &\n{}     & \n{j_1}   &  \n{}     & \n{j_2} & \n{}  \\ \hline
\n{}     &\n{.}    & \n{.}     &  \n{.}    & \n{.}              & \n{.}   \\
\n{i_1}  &\n.      & \n{t}     &  \n.      & \n{s}           & \n.   \\
\n{}     &\n{.}    & \n{.}     &  \n{.}    & \n{.}              & \n{.}  \\
\n{i_2}  &\n{.}     & \n{a}+\m{b}{t}  &  \n{.}     & \n{t}  & \n{.}  \\
\n{}     &\n{.}    & \n{.}     &  \n{.}    & \n{.}              & \n{.}  \\
\end{tabular}
\end{center}
and we are done.

{\bf Case 2}:  $i_2\neq i_3$.

It is easy to see that we can find the following cyclic pattern
lying in rows $i_2$ and $i_3$ of  $A$ (possibly after permuting
some columns of $A$)

\begin{center}
$A$ 
\def\arraystretch{0.8}
\begin{tabular}
{@{\hspace{7pt}}c@{\hspace{3pt}}|@{\hspace{3pt}}c@{\hspace{7pt}}
 @{\hspace{7pt}}c@{\hspace{7pt}}
 @{\hspace{7pt}}c@{\hspace{7pt}}
  @{\hspace{7pt}}c@{\hspace{7pt}}
  @{\hspace{7pt}}c@{\hspace{7pt}}
 @{\hspace{7pt}}c@{\hspace{7pt}}
 @{\hspace{7pt}}c@{\hspace{7pt}}@{\hspace{7pt}}c@{\hspace{7pt}}
 @{\hspace{7pt}}c@{\hspace{7pt}}
 @{\hspace{7pt}}c@{\hspace{7pt}}
}

\n{}     &\n{}     & \n{j_1}   &  \n{c_1}     &
\n{c_{2}}& \n{c_3}& \n{\dots} &\n{c_{r-1}}   &
\n{c_r}&\n{j_2}&\n{}\\ \hline

\n{}     &\n{.}    &\n{.}           &\n{.} &\n{.} &\n{.}&\n{\dots}   &\n{.}    &\n{.}     &\n{.} &\n{.}\\
\n{i_1}  &\n.      &\n{s}           &\n{c} & \n.  &\n{.}& \n{\dots}  & \n{.}   &  \n.     &\n{t} &\n.  \\
\n{}     &\n{.}    & \n{.}          &\n{.} & \n{.}&\n{.}& \n{\dots}  & \n{.}   &  \n{.}   &\n{.} &\n{.}\\
\n{i_2}  &\n{.}    &\n{a}+\m{b}{s}  &\n{s} & \n{u}&\n{v}& \n{\dots}  & \n{x}   &  \n{y}   &\n{.} &\n{.} \\
\n{}     &\n{.}    &\n{.}           &\n{.} & \n{.}&\n{.}& \n{\dots}  & \n{.}   &  \n{.}   &\n{.} &\n{.}\\
\n{i_3}  &\n{.}    &\n{d}           &\n{u} & \n{v}&\n{w}& \n{\dots}  & \n{y}   &  \n{z}   &\n{s} &\n{.}\\
\n{}     &\n{.}    &\n{.}           &\n{.} & \n{.}&\n{.}& \n{\dots}  & \n{.}   &  \n{.}   &\n{.} &\n{.}\\
\end{tabular}
\end{center}
where $\{u,v,w,\dots,x,y\}\cap\{a,b\}=\emptyset$, but
$z\in\{a,b\}$. Without loss of generality we assume that $z=b$.
Therefore $A$ has the following cyclic pattern

\begin{center}
$A$ \
\def\arraystretch{0.8}
\begin{tabular}
{@{\hspace{7pt}}c@{\hspace{3pt}}|@{\hspace{3pt}}c@{\hspace{7pt}}
 @{\hspace{7pt}}c@{\hspace{7pt}}
 @{\hspace{7pt}}c@{\hspace{7pt}}
  @{\hspace{7pt}}c@{\hspace{7pt}}
  @{\hspace{7pt}}c@{\hspace{7pt}}
 @{\hspace{7pt}}c@{\hspace{7pt}}
 @{\hspace{7pt}}c@{\hspace{7pt}}@{\hspace{7pt}}c@{\hspace{7pt}}
 @{\hspace{7pt}}c@{\hspace{7pt}}
 @{\hspace{7pt}}c@{\hspace{7pt}}
}

\n{}     &\n{}     & \n{j_1}   &  \n{c_1}     & \n{c_{2}}&
\n{c_3}& \n{\dots} &\n{c_{r-1}}   & \n{c_r}&\n{j_2}&\n{}\\ \hline

\n{}     &\n{.}    &\n{.}           &\n{.} &\n{.} &\n{.}&\n{\dots}   &\n{.}    &\n{.}     &\n{.} &\n{.}\\
\n{i_1}  &\n.      &\n{s}           &\n{c} & \n.  &\n{.}& \n{\dots}  & \n{.}   &  \n.     &\n{t} &\n.  \\
\n{}     &\n{.}    & \n{.}          &\n{.} & \n{.}&\n{.}& \n{\dots}  & \n{.}   &  \n{.}   &\n{.} &\n{.}\\
\n{i_2}  &\n{.}    &\n{a}+\m{b}{s}  &\n{s} & \n{u}&\n{v}& \n{\dots}  & \n{x}   &  \n{y}   &\n{.} &\n{.} \\
\n{}     &\n{.}    &\n{.}           &\n{.} & \n{.}&\n{.}& \n{\dots}  & \n{.}   &  \n{.}   &\n{.} &\n{.}\\
\n{i_3}  &\n{.}    &\n{d}           &\n{u} & \n{v}&\n{w}& \n{\dots}  & \n{y}   &  \n{b}   &\n{s} &\n{.}\\
\n{}     &\n{.}    &\n{.}           &\n{.} & \n{.}&\n{.}& \n{\dots}  & \n{.}   &  \n{.}   &\n{.} &\n{.}\\
\end{tabular}
\end{center}
%
%
Then the $((i_2,j_1;s),(i_1,c_1;b))$-move transfers $A$ to:
\begin{center}
$A_1$
\def\arraystretch{0.8}
\begin{tabular}
{@{\hspace{7pt}}c@{\hspace{3pt}}|@{\hspace{3pt}}c@{\hspace{7pt}}
 @{\hspace{7pt}}c@{\hspace{7pt}}
 @{\hspace{7pt}}c@{\hspace{7pt}}
  @{\hspace{7pt}}c@{\hspace{7pt}}
  @{\hspace{7pt}}c@{\hspace{7pt}}
 @{\hspace{7pt}}c@{\hspace{7pt}}
 @{\hspace{7pt}}c@{\hspace{7pt}}@{\hspace{7pt}}c@{\hspace{7pt}}
 @{\hspace{7pt}}c@{\hspace{7pt}}
 @{\hspace{7pt}}c@{\hspace{7pt}}
}

\n{}     &\n{}     & \n{j_1}   &  \n{c_1}     & \n{c_{2}}&
\n{c_3}& \n{\dots} &\n{c_{r-1}}   & \n{c_r}&\n{j_2}&\n{}\\ \hline

\n{}     &\n{.}    &\n{.}    &\n{.}         &\n{.} &\n{.}&\n{\dots}   &\n{.}    &\n{.}     &\n{.} &\n{.}\\
\n{i_1}  &\n.      &\n{b}    &\n{c}+\m{s}{b}& \n.  &\n{.}& \n{\dots}  & \n{.}   &  \n.     &\n{t} &\n.  \\
\n{}     &\n{.}    & \n{.}   &\n{.}         & \n{.}&\n{.}& \n{\dots}  & \n{.}   &  \n{.}   &\n{.} &\n{.}\\
\n{i_2}  &\n{.}    &\n{a}    &\n{b}         & \n{u}&\n{v}& \n{\dots}  & \n{x}   &  \n{y}   &\n{.} &\n{.} \\
\n{}     &\n{.}    &\n{.}    &\n{.}         & \n{.}&\n{.}& \n{\dots}  & \n{.}   &  \n{.}   &\n{.} &\n{.}\\
\n{i_3}  &\n{.}    &\n{d}    &\n{u}         & \n{v}&\n{w}& \n{\dots}  & \n{y}   &  \n{b}   &\n{s} &\n{.}\\
\n{}     &\n{.}    &\n{.}    &\n{.}         & \n{.}&\n{.}& \n{\dots}  & \n{.}   &  \n{.}   &\n{.} &\n{.}\\
\end{tabular}
\end{center}
If $u=b$ (i.e. $r=1$) then the $((i_3,j_1;b),(i_1,c_1;s))$-move
transfers $A_1$ to:
\begin{center}
\def\arraystretch{0.8}
\begin{tabular}
{@{\hspace{7pt}}c@{\hspace{7pt}}|@{\hspace{7pt}}c@{\hspace{7pt}}
 @{\hspace{7pt}}c@{\hspace{7pt}}
@{\hspace{7pt}}c@{\hspace{7pt}}@{\hspace{7pt}}c@{\hspace{7pt}}
@{\hspace{7pt}}c@{\hspace{7pt}}
} \n{}     &\n{}     & \n{j_1}   &    \n{c_1} &\n{j_2} &\n{} \\
\hline
\n{}     &\n{.}    & \n{.}           &  \n{.}    &\n{.} &\n{.}\\
\n{i_1}  &\n.      & \n{s}           & \n{c}     &\n{t}   &\n.  \\
\n{}     &\n{.}    & \n{.}           &\n{.}      &\n{.} &\n{.}\\
\n{i_2}  &\n{.}    & \n{a}           &\n{b}      &\n{.}  &\n{.} \\
\n{}     &\n{.}    & \n{.}           &\n{.}      &\n{.} &\n{.}\\
\n{i_3}  &\n{.}    & \n{d}+\m{b}{s}  &\n{s}      &\n{s} &\n{.}\\
\n{}     &\n{.}    & \n{.}           &\n{.}      &\n{.} &\n{.}\\
\end{tabular}
\end{center}
which reduces the problem to Case 1.
So we assume that $u\neq b$. Now the symbol $b$ appears once more
as a positive entry in column $c_1$ and another row, say $i_4$:
\begin{center}
$A_1$ 
\def\arraystretch{0.8}
\begin{tabular}
{@{\hspace{7pt}}c@{\hspace{3pt}}|@{\hspace{3pt}}c@{\hspace{7pt}}
 @{\hspace{7pt}}c@{\hspace{7pt}}
 @{\hspace{7pt}}c@{\hspace{7pt}}
  @{\hspace{7pt}}c@{\hspace{7pt}}
  @{\hspace{7pt}}c@{\hspace{7pt}}
 @{\hspace{7pt}}c@{\hspace{7pt}}
 @{\hspace{7pt}}c@{\hspace{7pt}}@{\hspace{7pt}}c@{\hspace{7pt}}
 @{\hspace{7pt}}c@{\hspace{7pt}}
 @{\hspace{7pt}}c@{\hspace{7pt}}
}

\n{}     &\n{}     & \n{j_1}   &  \n{c_1}     & \n{c_{2}}&
\n{c_3}& \n{\dots} &\n{c_{r-1}}   & \n{c_r}&\n{j_2}&\n{}\\ \hline

\n{}     &\n{.}    &\n{.}    &\n{.}         &\n{.} &\n{.}&\n{\dots}   &\n{.}    &\n{.}     &\n{.} &\n{.}\\
\n{i_1}  &\n.      &\n{b}    &\n{c}+\m{s}{b}& \n.  &\n{.}& \n{\dots}  & \n{.}   &  \n.     &\n{t} &\n.  \\
\n{}     &\n{.}    & \n{.}   &\n{.}         & \n{.}&\n{.}& \n{\dots}  & \n{.}   &  \n{.}   &\n{.} &\n{.}\\
\n{i_2}  &\n{.}    &\n{a}    &\n{b}         & \n{u}&\n{v}& \n{\dots}  & \n{x}   &  \n{y}   &\n{.} &\n{.} \\
\n{}     &\n{.}    &\n{.}    &\n{.}         & \n{.}&\n{.}& \n{\dots}  & \n{.}   &  \n{.}   &\n{.} &\n{.}\\
\n{i_3}  &\n{.}    &\n{d}    &\n{u}         & \n{v}&\n{w}& \n{\dots}  & \n{y}   &  \n{b}   &\n{s} &\n{.}\\
\n{}     &\n{.}    &\n{.}    &\n{.}         & \n{.}&\n{.}& \n{\dots}  & \n{.}   &  \n{.}   &\n{.} &\n{.}\\
\n{i_4}  &\n.      & \n{.}   &\n{b}         & \n{.}& \n. & \n{\dots}  &\n{.}    & \n{.}    &\n{.} &\n.  \\
\n{}     &\n{.}    &\n{.}    &  \n{.}       & \n{.}&\n{.}& \n{\dots}  & \n{.}   &\n{.}     &\n{.} &\n{.}\\
\end{tabular}
\end{center}

where $i_4 \notin\{i_1,i_2,i_3 \}$. By Lemma~\ref{improper-proper}, with a sequence of $\pm 1$-moves only on
rows $i_1$ and $i_4$, we can obtain a proper Latin square $A_2$.
Using Lemma~\ref{cycleswap}, a sequence of $\pm 1$-moves
transfers $A_2$ to
\begin{center}
$A_3$ 
\def\arraystretch{0.8}
\begin{tabular}
{@{\hspace{7pt}}c@{\hspace{7pt}}|@{\hspace{7pt}}c@{\hspace{7pt}}
 @{\hspace{7pt}}c@{\hspace{7pt}}
 @{\hspace{7pt}}c@{\hspace{7pt}}
  @{\hspace{7pt}}c@{\hspace{7pt}}
  @{\hspace{7pt}}c@{\hspace{7pt}}
 @{\hspace{7pt}}c@{\hspace{7pt}}
 @{\hspace{7pt}}c@{\hspace{7pt}}@{\hspace{7pt}}c@{\hspace{7pt}}
 @{\hspace{7pt}}c@{\hspace{7pt}}
 @{\hspace{7pt}}c@{\hspace{7pt}}
}

\n{}     &\n{}     & \n{j_1}   &  \n{c_1}     & \n{c_{2}}&
\n{c_3}& \n{\dots} &\n{c_{r-1}}   & \n{c_r}&\n{j_2}&\n{}\\ \hline

\n{}     &\n{.}    & \n{.}   &\n{.}         & \n{.}&\n{.}& \n{\dots}  & \n{.}   &  \n{.}   &\n{.} &\n{.}\\
\n{i_2}  &\n{.}    &\n{a}    &\n{u}         & \n{v}&\n{w}& \n{\dots}  & \n{y}   &  \n{b}   &\n{.} &\n{.}\\
\n{}     &\n{.}    &\n{.}    &\n{.}         & \n{.}&\n{.}& \n{\dots}  & \n{.}   &  \n{.}   &\n{.} &\n{.}\\
\n{i_3}  &\n{.}    &\n{d}    &\n{b}         & \n{u}&\n{v}& \n{\dots}  & \n{x}   &  \n{y}   &\n{s} &\n{.} \\
\n{}     &\n{.}    &\n{.}    &\n{.}         & \n{.}&\n{.}& \n{\dots}  & \n{.}   &  \n{.}   &\n{.} &\n{.}\\
\end{tabular}
\end{center}
Now we can undo the sequence of $\pm 1$-moves on rows $i_1$ and
$i_4$ in $A_3$, to obtain the corresponding rows in $A_1$. The
resulting Latin square has the following pattern \\
\begin{center}
$A_4$ 
\def\arraystretch{0.8}
\begin{tabular}
{@{\hspace{7pt}}c@{\hspace{3pt}}|@{\hspace{3pt}}c@{\hspace{7pt}}
 @{\hspace{7pt}}c@{\hspace{7pt}}
 @{\hspace{7pt}}c@{\hspace{7pt}}
  @{\hspace{7pt}}c@{\hspace{7pt}}
  @{\hspace{7pt}}c@{\hspace{7pt}}
 @{\hspace{7pt}}c@{\hspace{7pt}}
 @{\hspace{7pt}}c@{\hspace{7pt}}@{\hspace{7pt}}c@{\hspace{7pt}}
 @{\hspace{7pt}}c@{\hspace{7pt}}
 @{\hspace{7pt}}c@{\hspace{7pt}}
}

\n{}     &\n{}     & \n{j_1}   &  \n{c_1}     & \n{c_{2}}&
\n{c_3}& \n{\dots} &\n{c_{r-1}}   & \n{c_r}&\n{j_2}&\n{}\\ \hline

\n{}     &\n{.}    &\n{.}    &\n{.}         &\n{.} &\n{.}&\n{\dots}   &\n{.}    &\n{.}     &\n{.} &\n{.}\\
\n{i_1}  &\n.      &\n{b}    &\n{c}+\m{s}{b}& \n.  &\n{.}& \n{\dots}  & \n{.}   &  \n.     &\n{t} &\n.  \\
\n{}     &\n{.}    & \n{.}   &\n{.}         & \n{.}&\n{.}& \n{\dots}  & \n{.}   &  \n{.}   &\n{.} &\n{.}\\
\n{i_2}  &\n{.}    &\n{a}    &\n{u}         & \n{v}&\n{w}& \n{\dots}  & \n{y}   &  \n{b}   &\n{.} &\n{.} \\
\n{}     &\n{.}    &\n{.}    &\n{.}         & \n{.}&\n{.}& \n{\dots}  & \n{.}   &  \n{.}   &\n{.} &\n{.}\\
\n{i_3}  &\n{.}    &\n{d}    &\n{b}         & \n{u}&\n{v}& \n{\dots}  & \n{x}   &  \n{y}   &\n{s} &\n{.}\\
\n{}     &\n{.}    &\n{.}    &\n{.}         & \n{.}&\n{.}& \n{\dots}  & \n{.}   &  \n{.}   &\n{.} &\n{.}\\
\end{tabular}
\end{center}
With the  $((i_3,j_1;b),(i_1,c_1,s))$-move, $A_4$ can be
transfered to
\begin{center}
$A_5$ 
\def\arraystretch{0.8}
\begin{tabular}
{@{\hspace{7pt}}c@{\hspace{3pt}}|@{\hspace{3pt}}c@{\hspace{7pt}}
 @{\hspace{7pt}}c@{\hspace{7pt}}
 @{\hspace{7pt}}c@{\hspace{7pt}}
  @{\hspace{7pt}}c@{\hspace{7pt}}
  @{\hspace{7pt}}c@{\hspace{7pt}}
 @{\hspace{7pt}}c@{\hspace{7pt}}
 @{\hspace{7pt}}c@{\hspace{7pt}}@{\hspace{7pt}}c@{\hspace{7pt}}
 @{\hspace{7pt}}c@{\hspace{7pt}}
 @{\hspace{7pt}}c@{\hspace{7pt}}
}

\n{}     &\n{}     & \n{j_1}   &  \n{c_1}     & \n{c_{2}}&
\n{c_3}& \n{\dots} &\n{c_{r-1}}   & \n{c_r}&\n{j_2}&\n{}\\ \hline

\n{}     &\n{.}    &\n{.}         &\n{.}  &\n{.} &\n{.}&\n{\dots}   &\n{.}    &\n{.}     &\n{.} &\n{.}\\
\n{i_1}  &\n.      &\n{s}         &\n{c}  &\n.   &\n{.}& \n{\dots}  & \n{.}   &  \n.     &\n{t} &\n.  \\
\n{}     &\n{.}    & \n{.}        &\n{.}  & \n{.}&\n{.}& \n{\dots}  & \n{.}   &  \n{.}   &\n{.} &\n{.}\\
\n{i_2}  &\n{.}    &\n{a}         &\n{u}  & \n{v}&\n{w}& \n{\dots}  & \n{y}   &  \n{b}   &\n{.} &\n{.} \\
\n{}     &\n{.}    &\n{.}         &\n{.}  & \n{.}&\n{.}& \n{\dots}  & \n{.}   &  \n{.}   &\n{.} &\n{.}\\
\n{i_3}  &\n{.}    &\n{d}+\m{b}{s}&\n{s}  & \n{u}&\n{v}& \n{\dots}  & \n{x}   &  \n{y}   &\n{s} &\n{.}\\
\n{}     &\n{.}    &\n{.}         &\n{.}  & \n{.}&\n{.}& \n{\dots}  & \n{.}   &  \n{.}   &\n{.} &\n{.}\\
\end{tabular}
\end{center}
and finally Case 1 finishes the proof. Note that in row $i_1$,
 except the positions  of $s$ and $t$, other positions are unchanged.
\end{proof}

Now we can prove that the graph $G$ is connected.
\begin{theorem}
\label{algorithm} Let $S$ be the set of all proper or improper
Latin  squares of order $n$. Given two Latin  squares of order
$n$, there exists a sequence of $\pm 1$-moves that transfers one
square into the other without leaving $S$. An upper bound on the
length of the shortest such sequence is $2(n-1)^3$.
\end{theorem}
\begin{proof}
Suppose that $A$ and $B$ are two proper or improper Latin
squares. Without loss of generality we can assume that $A$ and $B$
are proper (see Lemma~\ref{improper-proper}).
To prove the theorem, we proceed by induction on the number of
identical rows of $A$ and $B$. Suppose that the first $k-1$ rows
of $A$ and $B$ are equal. We show that we can apply a sequence of
$\pm 1$-moves to $A$ to obtain a Latin square  with  the first $k$
rows identical to the first $k$ rows  of $B$. If the $k$th rows
of $A$ and $B$ are equal then we are done. So suppose that they
are not equal. In this case  we can find the following patterns
in $A$ and $B$ ($s\neq a$)

\begin{center}
$A=$ \
\def\arraystretch{0.8}
\begin{tabular}
{@{\hspace{3pt}}c@{\hspace{3pt}}|@{\hspace{3pt}}c@{\hspace{3pt}}
 @{\hspace{3pt}}c@{\hspace{3pt}}
 @{\hspace{3pt}}c@{\hspace{3pt}}
 @{\hspace{3pt}}c@{\hspace{3pt}}@{\hspace{3pt}}c@{\hspace{3pt}}
} \n{}    &\n{}     & \n{j_1}   &  \n{}     & \n{j_2}  & \n{} \\ \hline
\n{}     &\n{.}    & \n{.}     &  \n{.}    & \n{.}                 &\n{.} \\
\n{k}    &\n.        & \n{s}     &  \n.      & \n{t}            & \n.  \\
\n{}     &\n{.}    & \n{.}     &  \n{.}    & \n{.}                 &\n{.} \\
\n{i_1}  &\n{.}     & \n{u}  &  \n{.}     & \n{s}     &\n{.}\\
\n{}     &\n{.}    & \n{.}     &  \n{.}    & \n{.}                 &\n{.}
\end{tabular}
%
\hspace*{12mm}
$B=$
\def\arraystretch{0.8}
\begin{tabular}
{@{\hspace{3pt}}c@{\hspace{3pt}}|@{\hspace{3pt}}c@{\hspace{3pt}}
 @{\hspace{3pt}}c@{\hspace{3pt}}
 @{\hspace{3pt}}c@{\hspace{3pt}}
 @{\hspace{3pt}}c@{\hspace{3pt}}@{\hspace{3pt}}c@{\hspace{3pt}}
} \n{}     &\n{}     & \n{j_1}   &  \n{}     & \n{j_2} & \n{} \\ \hline
\n{}     &\n{.}    & \n{.}     &  \n{.}    & \n{.}             &\n{.} \\
\n{k}  &\n.      & \n{a}     &  \n.      & \n{s}              & \n.  \\
\n{}     &\n{.}    & \n{.}     &  \n{.}    & \n{.}             &\n{.} \\
\n{i_1}  &\n{.}     & \n{b}  &  \n{.}     & \n{c}   &\n{.}\\
\n{}     &\n{.}    & \n{.}     &  \n{.}    & \n{.}             &\n{.}
\end{tabular}
\end{center}
Since $A$ and $B$ are proper Latin squares and have the same first
$k-1$ rows, we must have $i_1> k$. Now the
$((k,j_2;s),(i_1,j_1;t))$-move transfers $A$ to
\begin{center}

\def\arraystretch{0.8}
\begin{tabular}
{@{\hspace{7pt}}c@{\hspace{7pt}}|@{\hspace{7pt}}c@{\hspace{7pt}}
 @{\hspace{7pt}}c@{\hspace{7pt}}
 @{\hspace{7pt}}c@{\hspace{7pt}}
 @{\hspace{7pt}}c@{\hspace{7pt}}@{\hspace{7pt}}c@{\hspace{7pt}}
} \n{}     &\n{}     & \n{j_1}   &  \n{}     & \n{j_2}  & \n{} \\ \hline
\n{}     &\n{.}    & \n{.}     &  \n{.}    & \n{.}                 &\n{.} \\
\n{k}  &\n.      & \n{t}     &  \n.      & \n{s}            & \n.  \\
\n{}     &\n{.}    & \n{.}     &  \n{.}    & \n{.}                 &\n{.} \\
\n{i_1}  &\n{.}     & \n{u}+\m{s}{t}  &  \n{.}     & \n{t}     &\n{.}\\
\n{}     &\n{.}    & \n{.}     &  \n{.}    & \n{.}                 &\n{.}\\

\end{tabular}
\end{center}

which fixes the position of $s$ in row $k$ in both squares.\\
 If $a=t$, then we
can find another entry in column $j_1$ of $A$ which is equal to
$t$ and is not in the first $k$ rows of $A$. So, by applying Lemma~\ref{improper-proper}, we can transfer the above (possibly
improper) Latin square into a proper Latin square (using at most
$\frac{n-1}{2}$, $\pm 1$-moves) without changing the first $k$
rows of $A$. \\
If $a\neq t$, then we can find the following
patterns in $B$   and $A$

\begin{center}
$A=$ \ \
\def\arraystretch{0.8}
\begin{tabular}
{@{\hspace{3pt}}c@{\hspace{3pt}}|@{\hspace{3pt}}c@{\hspace{3pt}}
 @{\hspace{3pt}}c@{\hspace{3pt}}
 @{\hspace{3pt}}c@{\hspace{3pt}}
 @{\hspace{3pt}}c@{\hspace{3pt}}@{\hspace{3pt}}c@{\hspace{3pt}}
} \n{}     &\n{}     & \n{j_1}   &  \n{}     & \n{j_3}  & \n{} \\
\hline
\n{}     &\n{.}    & \n{.}     &  \n{.}    & \n{.}                 &\n{.} \\
\n{k}  &\n.      & \n{t}     &  \n.      & \n{r}            & \n.  \\
\n{}     &\n{.}    & \n{.}     &  \n{.}    & \n{.}                 &\n{.} \\
\n{i_1}  &\n{.}     & \n{u}+\m{s}{t}  &  \n{.}     & \n{.}     &\n{.}\\
\n{}     &\n{.}    & \n{.}     &  \n{.}    & \n{.}                 &\n{.}\\
\n{i_2}     &\n{.}    & \n{.}     &  \n{.}    & \n{t}           &\n{.}\\
\end{tabular}
%
\hspace*{2mm}
$B=$ \
\def\arraystretch{0.8}
\begin{tabular}
{@{\hspace{3pt}}c@{\hspace{3pt}}|@{\hspace{3pt}}c@{\hspace{3pt}}
 @{\hspace{3pt}}c@{\hspace{3pt}}
 @{\hspace{3pt}}c@{\hspace{3pt}}
 @{\hspace{3pt}}c@{\hspace{3pt}}@{\hspace{3pt}}c@{\hspace{3pt}}
} \n{}     &\n{}     & \n{j_1}   &  \n{}     & \n{j_3} & \n{} \\
\hline
\n{}     &\n{.}    & \n{.}     &  \n{.}    & \n{.}             &\n{.} \\
\n{k}  &\n.      & \n{a}     &  \n.      & \n{t}              & \n.  \\
\n{}     &\n{.}    & \n{.}     &  \n{.}    & \n{.}             &\n{.} \\
\n{i_1}  &\n{.}     & \n{b}  &  \n{.}     & \n{.}   &\n{.}\\
\n{}     &\n{.}    & \n{.}     &  \n{.}    & \n{.}             &\n{.}\\
\n{i_2}     &\n{.}    & \n{.}     &  \n{.}    & \n{d}          &\n{.}\\
\end{tabular}
\end{center}
  Since the first $k-1$ rows of $A$ and $B$ are the same we
must have $i_2> k$. Therefore applying Lemma~\ref{ablemma},
interchanges $t$ and $r$ in row $k$ without any other changes in
row $k$ and the first $k-1$ rows. Applying this process (at most
$n-1$ times) produces a (proper or improper) Latin square $A'$
whose first $k$ rows are identical to those of $B$.  Using Lemma~\ref{improper-proper} (and the fact that $B$ is proper), we can
transfer $A'$ into a proper Latin square with a sequence of (at
most $\frac{n-1}{2}$) $\pm 1$-moves.  This finishes the proof by
induction.
 In order to transfer $A$ to $B$ we need to change
$n-1$ rows of $A$ and for each row we need at most $2(n-1)^2$,
$\pm 1$-moves. Therefore with at most $2(n-1)^3$,  $\pm 1$-moves
we can transfer $A$ to $B$.
\end{proof}

%
\pagebreak
\begin{remark}\textbf{Making moves ``properly''}

In ~\cite{MR1410617}, they introduce moves that stay within the
space of (proper) Latin squares. Such moves are called proper
moves. Using Theorem~\ref{algorithm} and a simple argument, they show that the space of
(proper) Latin squares is connected under these proper moves. So we just explain what
a proper move is in our notation. There are two kinds of proper moves, namely
``two-rowed proper moves'' and ``three-rowed proper moves''. In order to define them, we first
define the corresponding Latin bitrades.
A {\sf two-rowed Latin bitrade} is defined to be a Latin bitrade of the following form:
\begin{center}
\def\arraystretch{0.8}
\begin{tabular}
{@{\hspace{7pt}}c@{\hspace{7pt}}|@{\hspace{7pt}}c@{\hspace{7pt}}
 @{\hspace{7pt}}c@{\hspace{7pt}}
 @{\hspace{7pt}}c@{\hspace{7pt}}
  @{\hspace{7pt}}c@{\hspace{7pt}}
  @{\hspace{7pt}}c@{\hspace{7pt}}
 @{\hspace{7pt}}c@{\hspace{7pt}}
 @{\hspace{7pt}}c@{\hspace{7pt}}@{\hspace{7pt}}c@{\hspace{7pt}}
 @{\hspace{7pt}}c@{\hspace{7pt}}
 @{\hspace{7pt}}c@{\hspace{7pt}}
 @{\hspace{7pt}}c@{\hspace{7pt}}
 @{\hspace{7pt}}c@{\hspace{7pt}}
}

\n{}     &\n{}     & \n{j_1}   &  \n{j_2} & \n{\dots} & \n{j_{r-1}}& \n{j_{r}}
 &\n{} \\
\hline

\n{}     &\n{.} &\n{.}   &\n{.}     &\n{\dots} &\n{.}       &\n{.}   &\n.  \\
\n{i_1}  &\n.   &\m{a}{b}&\m{b}{c}  &\n{\dots} &\m{y}{z}    &\m{z}{a}&\n.  \\
\n{}     &\n{.} & \n{.}  &\n{.}     &\n{\dots} &\n{.}       &\n{.}   &\n{.}\\
\n{i_2}  &\n{.} &\m{b}{a}&\m{c}{b}  &\n{\dots} &\m{z}{y}    &\m{a}{z}&\n{.}\\
\n{}     &\n{.} &\n{.}   &\n{.}     & \n{\dots}&\n{.}       &\n{.}   &\n{.}\\

\end{tabular}
\end{center}
A {\sf three-rowed Latin bitrade} is a Latin bitrade $T$ with
the following properties:
\begin{enumerate}
\item $T$ has  exactly three nonempty
rows,
\item $T$ is the sum of two-rowed Latin bitrades $T_1$ and $T_2$ such that
there is at least one cell which is nonempty in both $T_1$ and $T_2$.
\end{enumerate}

Finally, a {\sf two-rowed proper move} (resp. {\sf three-rowed proper move})
means adding a
two-rowed Latin bitrade    (resp. three-rowed Latin bitrade)
to a given Latin square provided that the result is still a Latin square.

Another set of proper moves to connect the space of all Latin
squares which is similar to the ones found by Jacobson and Matthews, but certainly found independently, appears in Arthur O.
Pittenger~\cite{MR1448875}. Actually Pittenger's moves, 
 correspond to special kinds of
two-rowed and three-rowed moves, discussed above. 
\end{remark}
\begin{remark}
The Markov chain introduced in~\cite{MR1410617} is not known to
be rapidly mixing (and thus does not have proven efficiency).
Mark T. Jacobson and Peter Matthews~\cite{MR1410617} state that: ``in order to
use either of our Markov chains to generate almost-uniformly
distributed Latin squares, we must know how rapidly the chain
converges to the (uniform) stationary distribution. Of our two
chains, we suspect that the ``improper'' one mixes more rapidly,
in terms of real simulation time: executing a proper move takes
time comparable to that needed to execute an equivalent sequence
of $\pm 1$-moves; substituting an equal number of random $\pm
1$-moves seems likely to mix things up more.''

\end{remark}
\begin{remark}
Randomly generating combinatorial objects is an important problem
in combinatorics. It seems plausible to apply the ideas in this
paper to attack the same problem for some other combinatorial
objects such as STS's. In fact one can define the notion of an
improper
 STS see~\cite{MR2533146}.
\end{remark}
%
%

\textbf{Acknowledgements.} One of the authors (E.S.M.) thanks Amin
Saberi for hospitality of a short stay in Stanford University
when he introduced this problem. The other author (M.A.) was
partly supported by a grant from Sharif University of Technology
(Center of excellence in computational mathematics) and a grant
from IPM. We also thank  Masood Mortezaeefar for checking the
proofs and writing down a computer program for
Theorem~\ref{algorithm}.
%
\bibliographystyle{plain}
\bibliography{sr331Aryapoor4}
\end{document}